\documentclass[10pt, reqno]{amsart}
\usepackage{amssymb,amsmath,amsthm,tikz}
\usetikzlibrary{tikzmark,calc}

\usepackage{graphicx}
\usepackage{xcolor}

\usepackage[normalem]{ulem}

\usepackage{hyperref}

\theoremstyle{theorem}
\newtheorem{theorem}{Theorem}[section]

\newtheorem{corollary}[theorem]{Corollary}
\newtheorem{lemma}[theorem]{Lemma}

\theoremstyle{definition}

\theoremstyle{example}

\newcommand{\boundellipse}[3]
{(#1) ellipse (#2 and #3)
}

\allowdisplaybreaks

\title[]{Elementary Proofs and Generalizations of Recent Congruences of Thejitha and Fathima}

\author{James A. Sellers}
\address{Department of Mathematics and Statistics, 
University of Minnesota Duluth, Duluth, MN 55812, USA}
\email{jsellers@d.umn.edu}



\date{}

\begin{document}


\maketitle

\begin{abstract}
Motivated by recent work of Hirschhorn and the author, Thejitha and Fathima recently considered arithmetic properties satisfied by the function $a_5(n)$ which counts the number of integer partitions of weight $n$ wherein even parts come in only one color (i.e., they are monochromatic), while the odd parts may appear in one of five colors.  They proved two sets of Ramanujan--like congruences satisfied by $a_5(n)$, relying heavily on modular forms.  In this note, we prove their results via purely elementary means, utilizing generating function manipulations and elementary $q$-series dissections.  We then extensively generalize these two sets of congruences to infinite families of divisibility properties in which the results of Thejitha and Fathima are specific instances.  
\end{abstract}

\section{Introduction and background}
\label{sec:intro}
A partition of a positive integer $n$ is a finite sequence of positive integers $\lambda = (\lambda_1, \ldots, \lambda_j)$ with $\lambda_1 + \dots + \lambda_j = n$.
The $\lambda_i$, called the parts of $\lambda$, satisfy
\begin{equation*}
\lambda_1 \ge \cdots \ge \lambda_j.
\end{equation*}  
We denote the number of partitions of $n$ by $p(n)$; for example, the partitions of $n=4$ are 
$$(4), \ \ (3,1), \ \ (2,2), \ \ (2,1,1), \ \ (1,1,1,1),$$
and this implies that $p(4) = 5$.  

As an aside, we briefly highlight the work of Srinivasa Ramanujan on congruence properties satisfied by the partition function $p(n)$ \cite{Ram1919}.  In particular, Ramanujan proved that, for all $n\geq 0$, 
\begin{align}
p(5n+4) &\equiv 0 \pmod{5}, \label{RamCongMod5} \\
p(7n+5) &\equiv 0 \pmod{7}, \text{\ \ and} \notag \\
p(11n+6) &\equiv 0 \pmod{11}.  \notag
\end{align}

With the goal of generalizing recent work of Amdeberhan and Merca \cite{AM}, Hirschhorn and the author \cite{HS2025}  defined an infinite family of functions $a_k(n)$ as the number of partitions of $n$ wherein even parts come in only one color, while the odd parts may be ``colored'' with one of $k$ colors for fixed $k\geq 1$.    Clearly, $a_1(n) = p(n)$, the unrestricted integer partition function described above, while $a_2(n) = \overline{p}(n)$, the number of overpartitions of weight $n$ \cite{CL, HS05}, and $a_3(n) = a(n)$ of Amdeberhan and Merca \cite{AM}.  

It is straightforward to see that 
\begin{equation}
\label{genfn_general}
\sum_{n=0}^\infty a_k(n)q^n =\frac{1}{(q^2;q^2)_\infty (q;q^2)_\infty^k} = \frac{f_2^{k-1}}{f_1^k}
\end{equation}
where $f_k:=(q^k;q^k)_\infty$ and 
$$(A;q)_\infty := (1-A)(1-Aq)(1-Aq^2)(1-Aq^3)\dots $$ 
is the usual $q$--Pochhammer symbol.  

In \cite{HS2025}, Hirschhorn and the author proved the following family of congruences modulo 7 via elementary techniques.  

\begin{theorem}
\label{thm:infinite_families_mod7}
For all $j\geq 0$ and all $n\geq 0$, 
\begin{align*}
a_{7j+1}(7n+5) & \equiv 0 \pmod{7},  \\
a_{7j+3}(7n+2) & \equiv 0 \pmod{7}, \\
a_{7j+4}(7n+4) & \equiv 0 \pmod{7},  \\
a_{7j+5}(7n+6) & \equiv 0 \pmod{7}, \text{\ \ and}\\
a_{7j+7}(7n+3) & \equiv 0 \pmod{7}.  
\end{align*}
\end{theorem}

Motivated by these results, Thejitha and Fathima \cite{TF} recently proved the following two sets of congruences satisfied by the function $a_5(n)$.

\begin{theorem}
\label{thm:TF_mod5} 
For all $n\geq 0$, 
$$a_5(5n+3) \equiv 0 \pmod{5}.$$
\end{theorem}

\begin{theorem}
\label{thm:TF_mod3} 
For all $\alpha \geq 0$ and all $n\geq 0$, 
$$a_{5}\left(3^{2\alpha + 3}n + \frac{153\cdot 3^{2\alpha} - 1}{8}\right)  \equiv 0 \pmod{3}.$$
\end{theorem}

Our initial goal in this work is to provide elementary proofs of Theorems \ref{thm:TF_mod5} and \ref{thm:TF_mod3}.  This is in stark contrast to the work of Thejitha and Fathima \cite{TF} who relied heavily on modular forms.  

In Section \ref{sec:tools}, we collect the tools necessary for proving Theorem \ref{thm:TF_mod5} and Theorem \ref{thm:TF_mod3}.  In Section \ref{sec:proofs5}, we prove Theorem \ref{thm:TF_mod5} and show that it fits naturally into an infinite family of congruences modulo 5.  In Section \ref{sec:proofs3}, we prove Theorem \ref{thm:TF_mod3} as well as infinite families of congruences modulo 3 which are naturally related to the results of Thejitha and Fathima.  All of our proofs are elementary and follow from classic results in $q$--series along with straightforward generating function manipulations.  

\section{Necessary Tools}
\label{sec:tools}

Much of our work below on the family of congruences modulo 3 relies on results found in the paper of Hirschhorn and the author \cite{HS2005}.  Using the notation of \cite{HS2005}, let 
$$D(q) = \sum_{n=-\infty}^\infty (-1)^nq^{n^2}$$ 
and 
$$Y(q) = \sum_{n=-\infty}^\infty (-1)^nq^{3n^2-2n}.$$ 
Thanks to Jacobi's Triple Product Identity \cite[(1.1.1)]{Hir}, we know that 
\begin{equation}
\label{D_prod}
D(q) = \frac{f_1^2}{f_2} 
\end{equation}
and 
\begin{equation}
\label{Y_prod}
Y(q) = \frac{f_1f_6^2}{f_2f_3}.
\end{equation}
With these in hand, we can now state the 3--dissection results that we will require in order to prove Theorem \ref{thm:TF_mod3}.  

\begin{lemma}
\label{lem:f2overf1squared_3diss} 
We have 
$$
\frac{f_2}{f_1^2} \equiv  \frac{D(q^9)^2 +2qD(q^9)Y(q^3)+q^2Y(q^3)^2}{D(q^3)} \pmod{3}.
$$
\end{lemma}
\begin{proof}
See Hirschhorn and the author \cite{HS2005}.  
\end{proof}

\begin{lemma}
\label{lem:psi_3diss}
We have 
$$
\frac{f_2^2}{f_1} = \frac{f_6f_9^2}{f_3f_{18}} +q\frac{f_{18}^2}{f_9}.  
$$
\end{lemma}
\begin{proof}
See Hirschhorn \cite[(14.3.3)]{Hir}.
\end{proof}


To close this section, we note a pivotal congruence result which follows from the Binomial Theorem and congruence properties of certain binomial coefficients.  
\begin{lemma} 
\label{lem:FD} 
For a prime $p$ and positive integers $a$ and $b$, we have 
$$
f_a^{bp} \equiv f_{ap}^b \pmod{p}.  
$$
\end{lemma}

\section{Elementary Proofs of Theorem \ref{thm:TF_mod5} and a Natural Generalization}
\label{sec:proofs5}

In this brief section, we prove Theorem \ref{thm:TF_mod5} as well as an infinite generalization of the result.   

\begin{proof}[Proof of Theorem \ref{thm:TF_mod5}]
Note that 
\begin{align*}
\sum_{n=0}^\infty a_5(n)q^n 
&=
\frac{f_2^4}{f_1^5} 
= 
\frac{f_2^5}{f_1^5f_2} \\
&\equiv 
\frac{f_{10}}{f_{5}}\cdot \frac{1}{f_2} \pmod{5}\\
&= 
\frac{f_{10}}{f_{5}}\sum_{n=0}^\infty p(n)q^{2n}.
\end{align*}
In order to consider $a_5(5n+3)$ modulo 5 on the left--hand side above, we need to identify powers of $q$ on both sides which are congruent to $3 \pmod{5}$.  This means we need to have $2n\equiv 3 \pmod{5}$ which is equivalent to saying $n\equiv 4 \pmod{5}$.  Thus, on the right--hand side of the congruence above, each relevant term will have a coefficient which contains $p(5n+4)$ as a factor for some $n$.  Thanks to \eqref{RamCongMod5}, each of these values is divisible by 5, and this implies our result.  
\end{proof}

In a manner similar to that which was highlighted in \cite{HS2025}, we note that Theorem \ref{thm:TF_mod5} can be extended to an infinite family of results in the following way.  

\begin{corollary}
\label{cor:infinite_family_mod5}
For all $j\geq 0$ and all $n\geq 0$, 
\begin{align*}
a_{5j+5}(5n+3) &\equiv 0 \pmod{5}.
\end{align*}
\end{corollary}

\begin{proof}
Note that for any $j\geq 0$, 
\begin{align*}
\sum_{n=0}^\infty a_{5j+5}(n)q^n
&= 
\frac{f_2^{5j+4}}{f_1^{5j+5}} 
= 
\frac{f_2^{5j}}{f_1^{5j}}\cdot \frac{f_2^{4}}{f_1^{5}} \\
&\equiv 
\frac{f_{10}^{j}}{f_5^{j}}\cdot \frac{f_2^{4}}{f_1^{5}} \pmod{5} \\
&= 
\frac{f_{10}^{j}}{f_5^{j}}\sum_{n=0}^\infty a_{5}(n)q^n.
\end{align*}
The result then follows thanks to Theorem \ref{thm:TF_mod5} and the fact that $\frac{f_{10}^{j}}{f_5^{j}}$ is a function of $q^5$.  
\end{proof}

\section{Elementary Proofs of Theorem \ref{thm:TF_mod3} and Related Results}
\label{sec:proofs3}

We begin this section by providing an elementary proof of Theorem  \ref{thm:TF_mod3}.

\begin{proof}[Proof of Theorem  \ref{thm:TF_mod3}]
Thanks to \eqref{genfn_general} we know 
\begin{align*}
\sum_{n=0}^\infty a_5(n)q^n 
&= 
\frac{f_2^4}{f_1^5} = 
\frac{f_2^3}{f_1^3}\cdot \frac{f_2}{f_1^2} \\
&\equiv 
\frac{f_6}{f_3}\cdot \frac{f_2}{f_1^2} \pmod{3}  \\
&\equiv 
\frac{f_6}{f_3}\left( \frac{D(q^9)^2 +2qD(q^9)Y(q^3)+q^2Y(q^3)^2}{D(q^3)}  \right) \pmod{3} 
\end{align*}
using Lemma \ref{lem:f2overf1squared_3diss}.  
Thus, 
\begin{align*}
\sum_{n=0}^\infty a_5(3n+1)q^{3n+1} 
&\equiv 
2q\frac{f_6}{f_3}\left( \frac{D(q^9)Y(q^3)}{D(q^3)}  \right) \pmod{3}  \\
&= 
2q\frac{f_6f_9f_{18}}{f_3^2}
\end{align*}
using \eqref{D_prod} and \eqref{Y_prod} and simplifying the result.  This means that
\begin{align}
\sum_{n=0}^\infty a_5(3n+1)q^{n} 
&\equiv 
2f_3f_{6} \frac{f_2}{f_1^2} \pmod{3} \notag \\
&\equiv
2f_1f_2f_6 \pmod{3}.  \label{3n1_mod3}
\end{align}
Again using Lemma \ref{lem:f2overf1squared_3diss}, we see that 
\begin{align*}
\sum_{n=0}^\infty a_5(9n+1)q^{3n} 
&\equiv 
2f_3f_{6} \frac{D(q^9)^2}{D(q^3)} \pmod{3}
\end{align*}
or 
\begin{align*}
\sum_{n=0}^\infty a_5(9n+1)q^{n} 
&\equiv 
2f_1f_{2}\frac{D(q^3)^2}{D(q)} \pmod{3} \\
&= 
2f_1f_2\left( \frac{f_3^2}{f_6}\right)^2 \left( \frac{f_2}{f_1^2} \right)  = 
2\frac{f_3^4}{f_6^2}\left( \frac{f_2^2}{f_1} \right).
\end{align*}
Thanks to Lemma \ref{lem:psi_3diss}, we then know that 
\begin{align*}
\sum_{n=0}^\infty a_5(9n+1)q^{n} 
&\equiv 
2\frac{f_3^4}{f_6^2}\left( \frac{f_6f_9^2}{f_3f_{18}} +q\frac{f_{18}^2}{f_9} \right) \pmod{3}.
\end{align*}
Note that, when the above expression is expanded, it is not possible to obtain any terms involving powers of the form $q^{3n+2}$ for any $n$.  This means that, for all $n\geq 0$, 
\begin{equation*}
\label{27n19_mod3} 
a_5(9(3n+2)+1) = a_5(27n+19) \equiv 0 \pmod{3}. 
\end{equation*}
This is the $\alpha=0$ case of Theorem \ref{thm:TF_mod3}.  

We next wish to show that, for all $n\geq 0$, $a_5(81n+10)\equiv a_5(9n+1) \pmod{3}.$  Returning to the work we completed above, and noting that $a_5(9(3n+1)+1) = a_5(27n+10)$, we have 
\begin{align*}
\sum_{n=0}^\infty a_5(27n+10)q^{3n+1} 
&\equiv 
2\frac{f_3^4}{f_6^2}\left(q\frac{f_{18}^2}{f_9} \right) \pmod{3} \\
&\equiv 
2q\frac{f_3^4f_6^6}{f_3^3f_6^2}\pmod{3} \\
&= 
2qf_3f_6^4.
\end{align*}
Therefore, we know 
\begin{align*}
\sum_{n=0}^\infty a_5(27n+10)q^{n} 
&\equiv 
2f_1f_2^4 \pmod{3} \\
&\equiv 
2f_1f_2f_6 \pmod{3}.
\end{align*}
Thanks to \eqref{3n1_mod3}, we then see that, for all $n\geq 0$, 
$$
a_5(27n+10) \equiv a_5(3n+1) \pmod{3}.
$$
Replacing $n$ by $3n$ on both sides of this congruence yields $$
a_5(81n+10) \equiv a_5(9n+1) \pmod{3}
$$
and this is the congruence used by Thejitha and Fathima \cite{TF} to finalize their induction proof of Theorem \ref{thm:TF_mod3} for all $\alpha\geq 0$.  
\end{proof}

We now place Theorem \ref{thm:TF_mod3} within a larger context.  For the moment, we focus on the $\alpha=0$ case of the theorem, which states that, for all $n\geq 0$, $a_5(27n+19) \equiv 0 \pmod{3}$.  Interestingly, as noted above, $a_2(n) = \overline{p}(n)$, the number of overpartitions of weight $n$.  As proven in \cite[Theorem 2.1]{HS2005}, it is the case that,  for all $n\geq 0$, $a_2(27n+18) \equiv 0 \pmod{3}$.
It turns out that these two divisibility properties modulo 3 are specific examples of a more extensive theorem.  

\begin{theorem}
\label{thm:mod3_divisibilities}
For $0\leq t\leq 8$ and all $n\geq 0$, 
$$a_{3t+2}(27n+(18+t)) \equiv 0 \pmod{3}.$$ 
\end{theorem}
Note that the $t=0$ and $t=1$ cases of this theorem correspond to the congruences modulo 3 mentioned above for $a_2$ and $a_5$, respectively.  

\begin{proof}
Generally speaking, each of the eight proofs of the results above follows in a fashion similar to the proof given above for the congruence $a_5(27n+19) \equiv 0 \pmod{3}$.  

\bigskip  

\noindent 
{\it Proof that $a_{8}(27n+20) \equiv 0 \pmod{3}$}

\bigskip 

\noindent 
From our work above, we know 
\begin{align*}
\sum_{n=0}^\infty a_8(n)q^n 
&= 
\frac{f_2^7}{f_1^8} = 
\frac{f_2^6}{f_1^6}\cdot \frac{f_2}{f_1^2} \\
&\equiv 
\frac{f_{6}^2}{f_3^2}\cdot \frac{f_2}{f_1^2} \pmod{3}  \\
&\equiv 
\frac{f_6^2}{f_3^2}\left( \frac{D(q^9)^2 +2qD(q^9)Y(q^3)+q^2Y(q^3)^2}{D(q^3)}  \right) \pmod{3} 
\end{align*}
using Lemma \ref{lem:f2overf1squared_3diss}.  
Thus, 
\begin{align*}
\sum_{n=0}^\infty a_8(3n+2)q^{3n+2} 
&\equiv 
\frac{f_6^2}{f_3^2}\left( q^2\frac{Y(q^3)^2}{D(q^3)}  \right) \pmod{3}  \\
&= 
q^2\frac{f_6f_{18}^4}{f_3^2f_9^2}
\end{align*}
using \eqref{D_prod} and \eqref{Y_prod} and simplifying the result.  This means that
\begin{align*}
\sum_{n=0}^\infty a_8(3n+2)q^{n} 
&\equiv 
\frac{f_2f_{6}^4}{f_1^2f_3^2} \pmod{3}.
\end{align*}
Again using Lemma \ref{lem:f2overf1squared_3diss}, we see that 
\begin{align*}
\sum_{n=0}^\infty a_8(9n+2)q^{3n} 
&\equiv 
\frac{f_6^4}{f_3^2} \frac{D(q^9)^2}{D(q^3)} \pmod{3}
\end{align*}
or 
\begin{align*}
\sum_{n=0}^\infty a_8(9n+2)q^{n} 
&\equiv 
\frac{f_2^4}{f_1^2} \frac{D(q^3)^2}{D(q)} \pmod{3} \\
&= 
\frac{f_2^4}{f_1^2} \left(\frac{f_3^2}{f_6}\right)^2\frac{f_2}{f_1^2} 
= 
\frac{f_2^3f_3^4}{f_6^2f_1^3}\frac{f_2^2}{f_1} \\
&\equiv 
\frac{f_3^3}{f_6}\frac{f_2^2}{f_1} \pmod{3}.
\end{align*}
Thanks to Lemma \ref{lem:psi_3diss}, we then know that 
\begin{align*}
\sum_{n=0}^\infty a_8(9n+2)q^{n} 
&\equiv 
\frac{f_3^3}{f_6}\left( \frac{f_6f_9^2}{f_3f_{18}} +q\frac{f_{18}^2}{f_9} \right) \pmod{3}.
\end{align*}
As above, this means that, for all $n\geq 0$, 
\begin{equation*}
\label{27n19_mod3} 
a_8(9(3n+2)+2) = a_8(27n+20) \equiv 0 \pmod{3}. 
\end{equation*}

\bigskip  

\noindent
{\it Proof that $a_{11}(27n+21) \equiv 0 \pmod{3}$}

\bigskip 

\noindent 
We know 
\begin{align*}
\sum_{n=0}^\infty a_{11}(n)q^n 
&= 
\frac{f_2^{10}}{f_1^{11}} = 
\frac{f_2^9}{f_1^9}\cdot \frac{f_2}{f_1^2} \\
&\equiv 
\frac{f_{6}^3}{f_3^3}\cdot \frac{f_2}{f_1^2} \pmod{3}  \\
&\equiv 
\frac{f_6^3}{f_3^3}\left( \frac{D(q^9)^2 +2qD(q^9)Y(q^3)+q^2Y(q^3)^2}{D(q^3)}  \right) \pmod{3} 
\end{align*}
using Lemma \ref{lem:f2overf1squared_3diss}.  
Thus, 
\begin{align*}
\sum_{n=0}^\infty a_{11}(3n)q^{3n} 
&\equiv 
\frac{f_6^3}{f_3^3}\left( \frac{D(q^9)^2}{D(q^3)}  \right) \pmod{3}  \\
&= 
\frac{f_6^4f_9^4}{f_3^5f_{18}^2}
\end{align*}
using \eqref{D_prod} and simplifying the result.  This means that
\begin{align*}
\sum_{n=0}^\infty a_{11}(3n)q^{n} 
&\equiv 
\frac{f_2^4f_3^4}{f_1^5f_{6}^2} \pmod{3} \\
&\equiv 
\frac{f_3^3}{f_6}\frac{f_2}{f_1^2} \pmod{3}.
\end{align*}
Again using Lemma \ref{lem:f2overf1squared_3diss}, we see that 
\begin{align*}
\sum_{n=0}^\infty a_{11}(9n+3)q^{3n+1} 
&\equiv 
\frac{f_3^3}{f_6} \left( \frac{2qD(q^9)Y(q^3)}{D(q^3)} \right) \pmod{3}
\end{align*}
or 
\begin{align*}
\sum_{n=0}^\infty a_{11}(9n+3)q^{n} 
&\equiv 
2\frac{f_1^2f_3f_6}{f_2} \pmod{3} \\
&\equiv 
2f_3^2\frac{f_2^2}{f_1} \pmod{3}.
\end{align*}
Thanks to Lemma \ref{lem:psi_3diss}, and using the same logic as above, we see that, for all $n\geq 0$, 
\begin{equation*}
a_{11}(9(3n+2)+3) = a_{11}(27n+21) \equiv 0 \pmod{3}. 
\end{equation*}

\bigskip  

\noindent
{\it Proof that $a_{14}(27n+22) \equiv 0 \pmod{3}$}

\bigskip 

\noindent 
We know 
\begin{align*}
\sum_{n=0}^\infty a_{14}(n)q^n 
&= 
\frac{f_2^{13}}{f_1^{14}} = 
\frac{f_2^{12}}{f_1^{12}}\cdot \frac{f_2}{f_1^2} \\
&\equiv 
\frac{f_{6}^4}{f_3^4}\cdot \frac{f_2}{f_1^2} \pmod{3}  \\
&\equiv 
\frac{f_6^4}{f_3^4}\left( \frac{D(q^9)^2 +2qD(q^9)Y(q^3)+q^2Y(q^3)^2}{D(q^3)}  \right) \pmod{3} 
\end{align*}
using Lemma \ref{lem:f2overf1squared_3diss}.  
Thus, 
\begin{align*}
\sum_{n=0}^\infty a_{14}(3n+1)q^{3n+1} 
&\equiv 
\frac{f_6^4}{f_3^4}\left( \frac{2qD(q^9)Y(q^3)}{D(q^3)} \right) \pmod{3}  \\
&\equiv 
2q\frac{f_6^4f_9f_{18}}{f_3^5} \pmod{3}.
\end{align*}
This means that
\begin{align*}
\sum_{n=0}^\infty a_{14}(3n+1)q^{n} 
&\equiv 
2\frac{f_2^4f_3f_{6}}{f_1^5} \pmod{3} \\
&\equiv 
2f_6^2 \frac{f_2}{f_1^2} \pmod{3}.
\end{align*}
Again using Lemma \ref{lem:f2overf1squared_3diss}, we see that 
\begin{align*}
\sum_{n=0}^\infty a_{14}(9n+4)q^{3n+1} 
&\equiv 
2f_6^2 \left( \frac{2qD(q^9)Y(q^3)}{D(q^3)} \right) \pmod{3}
\end{align*}
or 
\begin{align*}
\sum_{n=0}^\infty a_{14}(9n+4)q^{n} 
&\equiv 
f_3f_6\frac{f_2^2}{f_1} \pmod{3}.
\end{align*}
Using the same logic as above, we see that, for all $n\geq 0$, 
\begin{equation*}
a_{14}(9(3n+2)+4) = a_{14}(27n+22) \equiv 0 \pmod{3}. 
\end{equation*}

\bigskip  

\noindent
{\it Proof that $a_{17}(27n+23) \equiv 0 \pmod{3}$}

\bigskip 

\noindent 
We know 
\begin{align*}
\sum_{n=0}^\infty a_{17}(n)q^n 
&= 
\frac{f_2^{16}}{f_1^{17}} = 
\frac{f_2^{15}}{f_1^{15}}\cdot \frac{f_2}{f_1^2} \\
&\equiv 
\frac{f_{6}^5}{f_3^5}\cdot \frac{f_2}{f_1^2} \pmod{3}  \\
&\equiv 
\frac{f_6^5}{f_3^5}\left( \frac{D(q^9)^2 +2qD(q^9)Y(q^3)+q^2Y(q^3)^2}{D(q^3)}  \right) \pmod{3} 
\end{align*}
using Lemma \ref{lem:f2overf1squared_3diss}.  
Thus, 
\begin{align*}
\sum_{n=0}^\infty a_{17}(3n+2)q^{3n+2} 
&\equiv 
\frac{f_6^5}{f_3^5}\left( \frac{q^2Y(q^3)^2}{D(q^3)} \right) \pmod{3}  \\
&\equiv 
q^2\frac{f_6^4f_{18}^4}{f_3^5f_9^2} \pmod{3}.
\end{align*}
This means that
\begin{align*}
\sum_{n=0}^\infty a_{17}(3n+2)q^{n} 
&\equiv 
\frac{f_2^4f_{6}^4}{f_1^5f_3^2} \pmod{3} \\
&\equiv 
\frac{f_6^5}{f_3^3}\frac{f_2}{f_1^2} \pmod{3}.
\end{align*}
Again using Lemma \ref{lem:f2overf1squared_3diss}, we see that 
\begin{align*}
\sum_{n=0}^\infty a_{17}(9n+5)q^{3n+1} 
&\equiv 
\frac{f_6^5}{f_3^3} \left( \frac{2qD(q^9)Y(q^3)}{D(q^3)} \right) \pmod{3}
\end{align*}
or 
\begin{align*}
\sum_{n=0}^\infty a_{17}(9n+5)q^{n} 
&\equiv 
2f_6^2\frac{f_2^2}{f_1} \pmod{3}.
\end{align*}
Using the same logic as above, we see that, for all $n\geq 0$, 
\begin{equation*}
a_{17}(9(3n+2)+5) = a_{17}(27n+23) \equiv 0 \pmod{3}. 
\end{equation*}

\bigskip  

\noindent
{\it Proof that $a_{20}(27n+24) \equiv 0 \pmod{3}$}

\bigskip 

\noindent 
We know 
\begin{align*}
\sum_{n=0}^\infty a_{20}(n)q^n 
&= 
\frac{f_2^{19}}{f_1^{20}} = 
\frac{f_2^{18}}{f_1^{18}}\cdot \frac{f_2}{f_1^2} \\
&\equiv 
\frac{f_{6}^6}{f_3^6}\cdot \frac{f_2}{f_1^2} \pmod{3}  \\
&\equiv 
\frac{f_6^6}{f_3^6}\left( \frac{D(q^9)^2 +2qD(q^9)Y(q^3)+q^2Y(q^3)^2}{D(q^3)}  \right) \pmod{3} 
\end{align*}
using Lemma \ref{lem:f2overf1squared_3diss}.  
Thus, 
\begin{align*}
\sum_{n=0}^\infty a_{20}(3n)q^{3n} 
&\equiv 
\frac{f_6^6}{f_3^6}\left( \frac{D(q^9)^2}{D(q^3)} \right) \pmod{3}  \\
&\equiv 
\frac{f_6^7f_{9}^4}{f_3^8f_{18}^2} \pmod{3}.
\end{align*}
This means that
\begin{align*}
\sum_{n=0}^\infty a_{20}(3n)q^{n} 
&\equiv 
\frac{f_2^7f_{3}^4}{f_1^8f_{6}^2} \pmod{3} \\
&\equiv 
f_3^2\frac{f_2}{f_1^2} \pmod{3}.
\end{align*}
Again using Lemma \ref{lem:f2overf1squared_3diss}, we see that 
\begin{align*}
\sum_{n=0}^\infty a_{20}(9n+6)q^{3n+2} 
&\equiv 
f_3^2 \left( \frac{q^2Y(q^3)^2}{D(q^3)} \right) \pmod{3}
\end{align*}
or 
\begin{align*}
\sum_{n=0}^\infty a_{20}(9n+6)q^{n} 
&\equiv 
\frac{f_6^3}{f_3}\frac{f_2^2}{f_1} \pmod{3}.
\end{align*}
Using the same logic as above, we see that, for all $n\geq 0$, 
\begin{equation*}
a_{20}(9(3n+2)+6) = a_{20}(27n+24) \equiv 0 \pmod{3}. 
\end{equation*}

\bigskip  

\noindent
{\it Proof that $a_{23}(27n+25) \equiv 0 \pmod{3}$}

\bigskip 

\noindent 
We know 
\begin{align*}
\sum_{n=0}^\infty a_{23}(n)q^n 
&= 
\frac{f_2^{22}}{f_1^{23}} = 
\frac{f_2^{21}}{f_1^{21}}\cdot \frac{f_2}{f_1^2} \\
&\equiv 
\frac{f_{6}^7}{f_3^7}\cdot \frac{f_2}{f_1^2} \pmod{3}  \\
&\equiv 
\frac{f_6^7}{f_3^7}\left( \frac{D(q^9)^2 +2qD(q^9)Y(q^3)+q^2Y(q^3)^2}{D(q^3)}  \right) \pmod{3} 
\end{align*}
using Lemma \ref{lem:f2overf1squared_3diss}.  
Thus, 
\begin{align*}
\sum_{n=0}^\infty a_{23}(3n+1)q^{3n+1} 
&\equiv 
\frac{f_6^7}{f_3^7}\left( \frac{2qD(q^9)Y(q^3)}{D(q^3)} \right) \pmod{3}  \\
&\equiv 
2q\frac{f_6^7f_{9}f_{18}}{f_3^8} \pmod{3}.
\end{align*}
This means that
\begin{align*}
\sum_{n=0}^\infty a_{23}(3n+1)q^{n} 
&\equiv 
2\frac{f_2^7f_{3}f_{6}}{f_1^8} \pmod{3} \\
&\equiv 
2\frac{f_6^3}{f_3}\frac{f_2}{f_1^2} \pmod{3}.
\end{align*}
Again using Lemma \ref{lem:f2overf1squared_3diss}, we see that 
\begin{align*}
\sum_{n=0}^\infty a_{23}(9n+7)q^{3n+2} 
&\equiv 
2\frac{f_6^3}{f_3} \left( \frac{q^2Y(q^3)^2}{D(q^3)} \right) \pmod{3}
\end{align*}
or 
\begin{align*}
\sum_{n=0}^\infty a_{23}(9n+7)q^{n} 
&\equiv 
2\frac{f_6^4}{f_3^2}\frac{f_2^2}{f_1} \pmod{3}.
\end{align*}
Using the same logic as above, we see that, for all $n\geq 0$, 
\begin{equation*}
a_{23}(9(3n+2)+7) = a_{23}(27n+25) \equiv 0 \pmod{3}. 
\end{equation*}

\bigskip  

\noindent
{\it Proof that $a_{26}(27n+26) \equiv 0 \pmod{3}$}

\bigskip 

\noindent 
We know 
\begin{align*}
\sum_{n=0}^\infty a_{26}(n)q^n 
&= 
\frac{f_2^{25}}{f_1^{26}} = 
\frac{f_2^{24}}{f_1^{24}}\cdot \frac{f_2}{f_1^2} \\
&\equiv 
\frac{f_{6}^8}{f_3^8}\cdot \frac{f_2}{f_1^2} \pmod{3}  \\
&\equiv 
\frac{f_6^8}{f_3^8}\left( \frac{D(q^9)^2 +2qD(q^9)Y(q^3)+q^2Y(q^3)^2}{D(q^3)}  \right) \pmod{3} 
\end{align*}
using Lemma \ref{lem:f2overf1squared_3diss}.  
Thus, 
\begin{align*}
\sum_{n=0}^\infty a_{26}(3n+2)q^{3n+2} 
&\equiv 
\frac{f_6^8}{f_3^8}\left( \frac{q^2Y(q^3)^2}{D(q^3)} \right) \pmod{3}  \\
&\equiv 
q^2\frac{f_6^7f_{18}^4}{f_3^8f_9^2} \pmod{3}.
\end{align*}
This means that
\begin{align*}
\sum_{n=0}^\infty a_{26}(3n+2)q^{n} 
&\equiv 
\frac{f_2^7f_{6}^4}{f_1^8f_3^2} \pmod{3} \\
&\equiv 
\frac{f_6^6}{f_3^4}\frac{f_2}{f_1^2} \pmod{3}.
\end{align*}
Again using Lemma \ref{lem:f2overf1squared_3diss}, we see that 
\begin{align*}
\sum_{n=0}^\infty a_{26}(9n+8)q^{3n+2} 
&\equiv 
\frac{f_6^6}{f_3^4} \left( \frac{q^2Y(q^3)^2}{D(q^3)} \right) \pmod{3}
\end{align*}
or 
\begin{align*}
\sum_{n=0}^\infty a_{26}(9n+8)q^{n} 
&\equiv 
\frac{f_6^5}{f_3^3}\frac{f_2^2}{f_1} \pmod{3}.
\end{align*}
Using the same logic as above, we see that, for all $n\geq 0$, 
\begin{equation*}
a_{26}(9(3n+2)+8) = a_{26}(27n+26) \equiv 0 \pmod{3}. 
\end{equation*}

\end{proof}

A priori, each of the congruences in Theorem \ref{thm:mod3_divisibilities} may serve as the basis case of an induction proof for an infinite family of congruences modulo 3 satisfied by the respective function (as in the proof of Theorem \ref{thm:TF_mod3}).  Of course, in order for such a family to be proved via induction, we require an internal congruence modulo 3 satisfied by each of the functions in question in order to complete the induction step in the proof.   
Indeed, we have the following: 

\begin{theorem}
For $0\leq t\leq 8$, and for all $n\geq 0$, 
$$
a_{3t+2}(27n+r_t) \equiv a_{3t+2}(3n+s_t) \pmod{3}
$$
where 
$$r_t = 
\begin{cases}
t & \text{if } t \text{ is even},\\
t+9 & \text{if } t \text{ is odd},
\end{cases}
$$
and 
$$s_t = 
\begin{cases}
0 & \text{if } t \text{ is even},\\
1 & \text{if } t \text{ is odd}.
\end{cases}
$$
\end{theorem}

\begin{proof}
We now provide detailed proofs for each of the eight results above.  

\bigskip  

\noindent
{\it Proof that $a_{2}(27n) \equiv a_{2}(3n) \pmod{3}$}

\bigskip 

\noindent 
This result is proven in \cite{HS2005}, keeping in mind that, for all $n\geq 0$, $a_2(n) = \overline{p}(n)$.  

\bigskip  

\noindent
{\it Proof that $a_{5}(27n+10) \equiv a_{5}(3n+1) \pmod{3}$}

\bigskip 

\noindent 
This result is proven above (and is a slightly stronger result than that proven by Thejitha and Fathima \cite{TF}).  

\bigskip  

\noindent
{\it Proof that $a_{8}(27n+2) \equiv a_{8}(3n) \pmod{3}$}

\bigskip 

\noindent 
Thanks to our previous work, we know 
\begin{align*}
\sum_{n=0}^\infty a_8(3n)q^{n} 
&\equiv 
\frac{f_2^3f_3^4}{f_1^4f_{6}^2} \pmod{3} \\
&\equiv 
\frac{f_1^8}{f_2^3} \pmod{3}.
\end{align*}
Moreover, in our work above, we noted that 
\begin{align*}
\sum_{n=0}^\infty a_8(9n+2)q^{n} 
&\equiv 
\frac{f_3^3}{f_6}\left( \frac{f_6f_9^2}{f_3f_{18}} +q\frac{f_{18}^2}{f_9} \right) \pmod{3}.
\end{align*}
Hence, 
\begin{align*}
\sum_{n=0}^\infty a_8(27n+2)q^{3n} 
&\equiv 
\frac{f_3^3}{f_6}\left( \frac{f_6f_9^2}{f_3f_{18}} \right) \pmod{3} \\
&= 
\frac{f_3^2f_9^2}{f_{18}} 
\end{align*}
so that 
\begin{align*}
\sum_{n=0}^\infty a_8(27n+2)q^{n} 
&\equiv 
\frac{f_1^2f_3^2}{f_{6}} \pmod{3} \\
&\equiv 
\frac{f_1^8}{f_{2}^3} \pmod{3} 
\end{align*}
and this yields our result.  

\bigskip  

\noindent
{\it Proof that $a_{11}(27n+12) \equiv a_{11}(3n+1) \pmod{3}$}

\bigskip 

\noindent 
Thanks to our previous work, we know 
\begin{align*}
\sum_{n=0}^\infty a_{11}(3n+1)q^{3n+1} 
&\equiv 
\frac{f_6^3}{f_3^3}\left( \frac{2qD(q^9)Y(q^3)}{D(q^3)}  \right) \pmod{3} \\
&=
2q\frac{f_6^3f_9f_{18}}{f_3^4}
\end{align*}
which means 
\begin{align*}
\sum_{n=0}^\infty a_{11}(3n+1)q^{n} 
&\equiv 
2\frac{f_2^3f_3f_6}{f_1^4} \pmod{3} \\
&\equiv 
2\frac{f_2^6}{f_1} \pmod{3}.  
\end{align*}
Moreover, in our work above, we noted that 
\begin{align*}
\sum_{n=0}^\infty a_{11}(9n+3)q^{n} 
&\equiv 
2\frac{f_1^2f_3f_6}{f_2} \pmod{3} \\
&\equiv 
2f_3^2\frac{f_2^2}{f_1} \pmod{3}.
\end{align*}
Hence, using Lemma \ref{lem:psi_3diss},
\begin{align*}
\sum_{n=0}^\infty a_{11}(27n+12)q^{3n+1} 
&\equiv 
2f_3^2\left( q\frac{f_{18}^2}{f_9}\right) \pmod{3} \\
&= 
2q\frac{f_3^2f_{18}^2}{f_9} 
\end{align*}
so that 
\begin{align*}
\sum_{n=0}^\infty a_{11}(27n+12)q^{n} 
&\equiv 
2\frac{f_1^2f_{6}^2}{f_3}  \pmod{3} \\
&\equiv 
2\frac{f_2^6}{f_{1}} \pmod{3} 
\end{align*}
and this yields our result.  

\bigskip  

\noindent
{\it Proof that $a_{14}(27n+4) \equiv a_{14}(3n) \pmod{3}$}

\bigskip 

\noindent 
Thanks to our previous work, we know 
\begin{align*}
\sum_{n=0}^\infty a_{14}(3n)q^{3n} 
&\equiv 
\frac{f_6^4}{f_3^4}\left( \frac{D(q^9)^2}{D(q^3)}  \right) \pmod{3} \\
&=
\frac{f_6^5f_9^4}{f_3^6f_{18}^2}
\end{align*}
which means 
\begin{align*}
\sum_{n=0}^\infty a_{14}(3n)q^{n} 
&\equiv 
\frac{f_2^5f_3^4}{f_1^6f_{6}^2} \pmod{3} \\
&\equiv 
\frac{f_1^6}{f_2} \pmod{3}.  
\end{align*}
Moreover, in our work above, we noted that 
\begin{align*}
\sum_{n=0}^\infty a_{14}(9n+4)q^{n} 
&\equiv 
f_3f_6\frac{f_2^2}{f_1} \pmod{3}.
\end{align*}
Hence, using Lemma \ref{lem:psi_3diss},
\begin{align*}
\sum_{n=0}^\infty a_{14}(27n+4)q^{3n} 
&\equiv 
f_3f_6\left( \frac{f_6f_9^2}{f_3f_{18}}  \right) \pmod{3} \\
&= 
\frac{f_6^2f_9^2}{f_{18}} 
\end{align*}
so that 
\begin{align*}
\sum_{n=0}^\infty a_{14}(27n+4)q^{n} 
&\equiv 
\frac{f_2^2f_{3}^2}{f_6}  \pmod{3} \\
&\equiv 
\frac{f_1^6}{f_{2}} \pmod{3} 
\end{align*}
and this yields our result.  

\bigskip  

\noindent
{\it Proof that $a_{17}(27n+14) \equiv a_{17}(3n+1) \pmod{3}$}

\bigskip 

\noindent 
Thanks to our previous work, we know 
\begin{align*}
\sum_{n=0}^\infty a_{17}(3n+1)q^{3n+1} 
&\equiv 
\frac{f_6^5}{f_3^5}\left( \frac{2qD(q^9)Y(q^3)}{D(q^3)} \right) \pmod{3} \\
&\equiv 
2q\frac{f_6^5f_9f_{18}}{f_3^6} \pmod{3}
\end{align*}
which means 
\begin{align*}
\sum_{n=0}^\infty a_{17}(3n+1)q^{n} 
&\equiv 
2\frac{f_2^8}{f_1^3} \pmod{3}.  
\end{align*}
From our earlier work, we know
\begin{align*}
\sum_{n=0}^\infty a_{17}(9n+5)q^{n} 
&\equiv 
2f_6^2\frac{f_2^2}{f_1} \pmod{3}.
\end{align*}
Thus, 
\begin{align*}
\sum_{n=0}^\infty a_{17}(27n+14)q^{3n+1} 
&\equiv 
2f_6^2\left(  q\frac{f_{18}^2}{f_9} \right) \pmod{3} 
\end{align*}
which implies that 
\begin{align*}
\sum_{n=0}^\infty a_{17}(27n+14)q^{n} 
&\equiv 
2\frac{f_2^2f_6^2}{f_3}  \pmod{3}  \\
&\equiv 
2\frac{f_2^8}{f_1^3}  \pmod{3}.
\end{align*}

\bigskip  

\noindent
{\it Proof that $a_{20}(27n+6) \equiv a_{20}(3n) \pmod{3}$}

\bigskip 

\noindent 
Thanks to our previous work, we know 
\begin{align*}
\sum_{n=0}^\infty a_{20}(3n)q^{n} 
&\equiv 
f_3^2\frac{f_2}{f_1^2} \pmod{3} \\
&\equiv 
f_1^4f_2 \pmod{3}.
\end{align*}
Moreover, we also determined above that 
\begin{align*}
\sum_{n=0}^\infty a_{20}(9n+6)q^{n} 
&\equiv 
\frac{f_6^3}{f_3}\frac{f_2^2}{f_1} \pmod{3}.
\end{align*}
Using Lemma \ref{lem:psi_3diss}, we then know that 
\begin{align*}
\sum_{n=0}^\infty a_{20}(27n+6)q^{3n} 
&\equiv 
\frac{f_6^3}{f_3}\left(  \frac{f_6f_9^2}{f_3f_{18}} \right) \pmod{3} \\
&= 
\frac{f_6^4f_9^2}{f_3^2f_{18}}  
\end{align*}
which yields 
\begin{align*}
\sum_{n=0}^\infty a_{20}(27n+6)q^{n} 
&\equiv  
\frac{f_2^4f_3^2}{f_1^2f_{6}}   \pmod{3} \\
&\equiv 
f_1^4f_2 \pmod{3}
\end{align*}
and this yields our result.  

\bigskip  

\noindent
{\it Proof that $a_{23}(27n+16) \equiv a_{23}(3n+1) \pmod{3}$}

\bigskip 

\noindent 
Thanks to our previous work, we know 
\begin{align*}
\sum_{n=0}^\infty a_{23}(3n+1)q^{n} 
&\equiv 
2\frac{f_6^3}{f_3}\frac{f_2}{f_1^2} \pmod{3} \\
&\equiv 
2\frac{f_2^{10}}{f_1^5} \pmod{3}.  
\end{align*}
Moreover, we also showed that 
\begin{align*}
\sum_{n=0}^\infty a_{23}(9n+7)q^{n} 
&\equiv 
2\frac{f_6^4}{f_3^2}\frac{f_2^2}{f_1} \pmod{3}
\end{align*}
which means 
\begin{align*}
\sum_{n=0}^\infty a_{23}(27n+16)q^{3n+1} 
&\equiv 
2\frac{f_6^4}{f_3^2}\left(  q\frac{f_{18}^2}{f_9}  \right) \pmod{3}
\end{align*}
or 
\begin{align*}
\sum_{n=0}^\infty a_{23}(27n+16)q^{n} 
&\equiv 
2\frac{f_2^4f_{6}^2}{f_1^2f_3}   \pmod{3} \\
&\equiv 
2\frac{f_2^{10}}{f_1^5} \pmod{3}.
\end{align*}
This proves our result.  

\bigskip  

\noindent
{\it Proof that $a_{26}(27n+8) \equiv a_{26}(3n) \pmod{3}$}

\bigskip 

\noindent 
Thanks to our previous work, we know 
\begin{align*}
\sum_{n=0}^\infty a_{26}(n)q^{3n} 
&\equiv 
\frac{f_6^8}{f_3^8}\left( \frac{D(q^9)^2}{D(q^3)}  \right) \pmod{3}  \\
&=
\frac{f_6^9f_9^4}{f_3^{10}f_{18}^2}
\end{align*}
which means 
\begin{align*}
\sum_{n=0}^\infty a_{26}(n)q^{n} 
&\equiv 
\frac{f_2^9f_3^4}{f_1^{10}f_{6}^2} \pmod{3} \\
&\equiv 
f_1^2f_2^3 \pmod{3}.  
\end{align*}
Also from our work above, we know 
\begin{align*}
\sum_{n=0}^\infty a_{26}(9n+8)q^{n} 
&\equiv 
\frac{f_6^5}{f_3^3}\frac{f_2^2}{f_1} \pmod{3}
\end{align*}
which implies
\begin{align*}
\sum_{n=0}^\infty a_{26}(27n+8)q^{3n} 
&\equiv 
\frac{f_6^5}{f_3^3}\left(  \frac{f_6f_9^2}{f_3f_{18}} \right) \pmod{3} \\
\end{align*}
or
\begin{align*}
\sum_{n=0}^\infty a_{26}(27n+8)q^{n} 
&\equiv 
\frac{f_2^6f_3^2}{f_1^4f_{6}} \pmod{3} \\
&\equiv 
f_1^2f_2^3 \pmod{3}.  
\end{align*}
This completes our proof.  
\end{proof}

Thanks to the above work, we can now state two new infinite families of congruences modulo 3 satisfied by two of the functions in this set.  

\begin{corollary}
\label{cor:new_inf_families_mod3}
For all $\alpha \geq 0$ and all $n\geq 0$, 
\begin{align}
a_{20}\left(3^{2\alpha + 3}n + \frac{198\cdot 3^{2\alpha} - 6}{8}\right)  &\equiv 0 \pmod{3}, \text{\ and} \label{a20_inf_fam} \\
a_{23}\left(3^{2\alpha + 3}n + \frac{207\cdot 3^{2\alpha} - 7}{8}\right)  &\equiv 0 \pmod{3}. \label{a23_inf_fam} 
\end{align}
\end{corollary}

\begin{proof}
Each of the above is proved via induction (in a manner very similar to the proof of Theorem \ref{thm:TF_mod3}).  Namely, the base case for the congruence family \eqref{a20_inf_fam} states that, for all $n\geq 0$, $a_{20}(27n+24) \equiv 0 \pmod{3}$.  This fact was proven above.  Moreover, the induction step follows from the fact that, for all $n\geq 0$, 
$$a_{20}(27n+6) \equiv a_{20}(3n) \pmod{3}$$ 
which was also proven above.  Similarly, \eqref{a23_inf_fam} follows from the fact that, for all $n\geq 0$, $a_{23}(27n+25)\equiv 0 \pmod{3}$ and 
$$a_{23}(27n+16) \equiv a_{23}(3n+1) \pmod{3},$$ 
both of which were proven above.  
\end{proof}
\noindent
We close this work by noting that Theorem \ref{thm:mod3_divisibilities} can be easily generalized to an infinite family of results.  

\begin{corollary}
\label{mod3_fullgeneralization}
For $j\geq 0$, $0\leq t\leq 8$, and all $n\geq 0$, 
$$a_{27j+3t+2}(27n+(18+t)) \equiv 0 \pmod{3}.$$ 
\end{corollary}

\begin{proof}
For any $j\geq 0$, 
\begin{align*}
\sum_{n=0}^\infty a_{27j+3t+2}(n)q^n
&= 
\frac{f_2^{27j+3t+1}}{f_1^{27j+3t+2}} 
= 
\frac{f_2^{27j}}{f_1^{27j}}\cdot \frac{f_2^{3t+1}}{f_1^{3t+2}} \\
&\equiv 
\frac{f_{54}^{j}}{f_{27}^{j}}\cdot \frac{f_2^{3t+1}}{f_1^{3t+2}} \pmod{3} \\
&= 
\frac{f_{54}^{j}}{f_{27}^{j}}\sum_{n=0}^\infty a_{3t+2}(n)q^n.
\end{align*}
The result then follows thanks to Theorem \ref{thm:mod3_divisibilities} and the fact that $\frac{f_{54}^{j}}{f_{27}^{j}}$ is a function of $q^{27}$.  
\end{proof}

\end{document}